\documentclass[12pt]{amsart}%
\usepackage{graphicx}
\usepackage{amscd, color}
\usepackage{amsmath}
\usepackage{amsfonts}
\usepackage{amssymb}%
\providecommand{\U}[1]{\protect\rule{.1in}{.1in}}
\providecommand{\U}[1]{\protect\rule{.1in}{.1in}}
\providecommand{\U}[1]{\protect\rule{.1in}{.1in}}
\theoremstyle{plain}

\newtheorem{theorem}{Theorem}[section]

\newtheorem{definition}[theorem]{Definition}
\numberwithin{equation}{section}

\begin{document}
\title{A unified Pietsch domination theorem}
\author{Geraldo Botelho, Daniel Pellegrino and Pilar Rueda}
\thanks{The second author is supported by CNPq Grants 471054/2006-2 and 308084/2006-3.
The third author is supported by MEC and FEDER Project MTM2005-08210.}

\begin{abstract}
In this paper we prove an abstract version of Pietsch's domination theorem which unify a number of known
Pietsch-type domination theorems for classes of mappings that generalize the ideal of absolutely $p$-summing
linear operators. A final result shows that Pietsch-type dominations are totally free from algebraic
conditions, such as linearity, multilinearity, etc.

\end{abstract}
\maketitle


\section{Introduction}




Pietsch's domination theorem (PDT) is a cornerstone in the theory of absolutely summing linear operators.
In this paper we prove an abstract version of PDT which has a twofold purpose.\\
\indent On the one hand, as expected, a Pietsch-type domination theorem turned out to be a basic result in
the several (linear and non-linear) theories which generalize and extend the linear theory of absolutely
summing operators. The canonical approach is as follows: (i) a class of (linear or non-linear) mappings
between (normed, metric) spaces is defined, (ii) a Pietsch-type domination theorem is proved, (iii) when
restricted to the class of linear operators the class defined in (i) is proved to coincide with the ideal of
absolutely $p$-summing linear operators and the theorem proved in (ii) recovers the classical PDT, (iv) the
new class is studied on its own and often compared with other related classes. Our point is that these
several PDTs are always proved using the same argument. In section 2 we construct an abstract setting where
this canonical argument yields a unified version of PDT which comprises, as we show in section 3, all the
known (as far as we know) domination theorems as particular cases.
\newline\indent The idea behind this first purpose is to avoid the repetition of the canonical argument whenever a new class is
introduced. Among the results that are unified by our main theorem (Theorem \ref{ds}) we mention: the classic
PDT for absolutely summing linear operators \cite[Theorem 2.12]{DJT}, the Farmer and Johnson domination
theorem for Lipschitz summing mappings between metric spaces \cite[Theorem 1(2)]{FaJo}, the Pietsch and Geiss
domination theorem for dominated multilinear mappings (\cite[Theorem 14]{Pie}, \cite[Satz 3.2.3]{Geiss}), the
Dimant domination theorem for strongly summing multilinear mappings and homogeneous polynomials
\cite[Proposition 1.2(ii) and Proposition 3.2(ii)]{Dimant}, the domination theorem for
$\alpha$-subhomogeneous mappings \cite[Theorem 2.4]{BPR}, the Martin\'{e}z-Gim\'{e}nez and S\'{a}nchez
P\'{e}rez domination theorem for $(D,p)$-summing operators \cite[Theorem 3.11]{SP}, etc.\\
\indent On the other hand, we note that several of the known domination theorems rely on certain algebraic
conditions of the mappings in question, such as linear operators, multilinear mappings, homogeneous
polynomials, etc. Our second purpose is to show that the Pietsch-type domination does not depend on any
algebraic condition. In fact, in section 4 we introduce the notion of absolutely summing arbitrary mapping
just mimicking the linear definition and we apply our unified PDT proved in section 2 to show that, even for
arbitrary mappings, being absolutely summing is equivalent to satisfy a PDT. In summary, this second purpose
is to show that Pietsch-type dominations are algebra-free.

\section{The main result}

Let $X$, $Y$ and $E$ be (arbitrary) sets, $\mathcal{H}$ be a family of
mappings from $X$ to $Y$, $G$ be a Banach space and $K$ be a compact Hausdorff
topological space. Let
\[
R \colon K\times E \times G\longrightarrow\lbrack0,\infty)~\mathrm{and~} S
\colon{\mathcal{H}}\times E \times G\longrightarrow\lbrack0,\infty)
\]
be mappings so that

\noindent$\bullet$ There is $x_{0} \in E$ such that
\[
R(\varphi, x_{0}, b) = S(f, x_{0}, b) = 0
\]
for every $\varphi\in K$ and $b \in G$.\newline\noindent$\bullet$ The mapping
\[
R_{x,b} \colon K \longrightarrow[0, \infty)~,~R_{x,b}(\varphi) =
R(\varphi,x,b)
\]
is continuous for every $x \in E$ and $b \in G$.\newline\noindent$\bullet$ It
holds that
\[
R\left(  \varphi,x,\eta b\right)  \leq\eta R\left(  \varphi,x,b\right)
~\mathrm{and~} \eta S(f,x,b) \leq S(f,x,\eta b)
\]
for every $\varphi\in K, x \in E, 0 \leq\eta\leq1, b\in G$ and $f
\in{\mathcal{H}}$.

\begin{definition}\rm
\textrm{\textrm{Let $R$ and $S$ be as above and $0<p<\infty$. A mapping
$f\in{\mathcal{H}}$ is said to be \textit{$R$-$S$-abstract $p$-summing} is
there is a constant $C_{1}>0$ so that%
\begin{equation}
\sum_{j=1}^{m}S(f,x_{j},b_{j})^{p}\leq C_{1}\sup_{\varphi\in K}\sum_{j=1}%
^{m}R\left(  \varphi,x_{j},b_{j}\right)  ^{p},\label{cam-errado}%
\end{equation}
for all $x_{1},\ldots,x_{m}\in E,$ $b_{1},\ldots,b_{m}\in G$ and
$m\in\mathbb{N}$. The infimum of such constants $C_{1}$ is denoted by
$\pi_{RS,p}(f)$. }}
\end{definition}

It is not difficult to show that the infimum of the constants above is
attained, i.e., $\pi_{RS,p}(f)$ satisfies (\ref{cam-errado}).\textrm{ }

\begin{theorem}
\label{Unified} \label{ds}Let $R$ and $S$ be as above, $0<p<\infty$ and
$f\in{\mathcal{H}}$. Then $f$ is $R$-$S$-abstract $p$-summing if and only if
there is a constant $C>0$ and a regular Borel probability measure $\mu$ on $K$
such that%
\begin{equation}
S(f,x,b)\leq C\left(  \int_{K}R\left(  \varphi,x,b\right)  ^{p}d\mu\left(
\varphi\right)  \right)  ^{\frac{1}{p}} \label{2}%
\end{equation}
for all $x\in E$ and $b\in G.$ Moreover, the infimum of such constants $C$
equals $\pi_{RS,p}(f)^{\frac{1}{p}}$.
\end{theorem}

\begin{proof}
Assume the existence of such a measure $\mu$. Then, given $m\in\mathbb{N}$,
$x_{1},\ldots,x_{m}\in E$ and $b_{1},\ldots,b_{m}\in G$,
\begin{align*}
\sum_{j=1}^{m}S(f,x_{j},b_{j})^{p}  &  \leq C^{p}\sum_{j=1}^{m}\int
_{K}R\left(  \varphi,x_{j},b_{j}\right)  ^{p}d\mu\left(  \varphi\right) \\
&  =C^{p}\int_{K}\sum_{j=1}^{m}R\left(  \varphi,x_{j},b_{j}\right)  ^{p}%
d\mu\left(  \varphi\right) \\
&  \leq C^{p}\sup_{\varphi\in K}\sum_{j=1}^{m}R\left(  \varphi,x_{j}%
,b_{j}\right)  ^{p}.
\end{align*}
Hence $f$ is $R$-$S$-abstract $p$-summing with $\pi_{RS,p}(f)\leq C^{p}.$

Conversely, suppose that $f\colon E\longrightarrow F$ is $R$-$S$-abstract
$p$-summing. Consider the Banach space $C(K)$ of continuous real-valued
functions on $K.$ For every finite set $M=\{(x_{1},b_{1}),\ldots,(x_{k}%
,b_{k})\}\subset E\times G,$ let%
\[
\Psi_{M}\colon K\longrightarrow\mathbb{R}~,~\Psi_{M}\left(  \varphi\right)
=\sum_{(x,b)\in M}\left(  S(f,x,b)^{p}-\pi_{RS,p}(f)R\left(  \varphi
,x,b\right)  ^{p}\right)  .
\]
It is convenient to regard $M$ as a finite sequence of elements of $E\times G$
rather than a finite set (that is, repetitions are allowed). Since the
functions $R_{x,b}\colon K\longrightarrow\lbrack0,\infty)$, $R_{x,b}%
(\varphi)=R(\varphi,x,b)$, are continuous, it is plain that $\Psi_{M}\in
C\left(  K\right)  $.\newline\indent Let $\mathcal{G}$ be the set of all such
$\Psi_{M}$ and $\mathcal{F}$ be its convex hull. Let us show that for every
$\Psi\in\mathcal{F}$ there is $\varphi_{\Psi}\in K$ such that $\Psi\left(
\varphi_{\Psi}\right)  \leq0$: given $\Psi\in\mathcal{F}$, there are
$k\in\mathbb{N}$, $\lambda_{1},\ldots,\lambda_{k}\in\lbrack0,1]$, $\sum
_{j=1}^{k}\lambda_{j}=1$, and $\Psi_{M_{1}},\ldots,\Psi_{M_{k}}\in
{\mathcal{G}}$ so that $\Psi=\sum_{j=1}^{k}\lambda_{j}\Psi_{M_{j}}$.

Define
\[
M_{\Psi}=\bigcup_{j=1}^{k}\left\{  \left(  x,\lambda_{j}^{\frac{1}{p}%
}b\right)  ;\text{ }(x,b)\in M_{j}\right\}
\]
and choose $\varphi_{\Psi}\in K$ so that
\[
\sup_{\varphi\in K}\sum_{(x,b)\in M_{\Psi}}R\left(  \varphi,x,b\right)
^{p}=\sum_{(x,b)\in M_{\Psi}}R\left(  \varphi_{\Psi},x,b\right)  ^{p}.
\]
Notice that such a $\varphi_{\Psi}$ exists since $K$ is compact and $R_{x,b}$
is continuous. Then
\begin{align*}
\Psi\left(  \varphi_{\Psi}\right)   &  =\sum_{j=1}^{k}\lambda_{j}\Psi_{M_{j}%
}\left(  \varphi_{\Psi}\right) \\
&  =\sum_{j=1}^{k}\lambda_{j}\sum_{(x,b)\in M_{j}}\left(  S(f,x,b)^{p}%
-\pi_{RS,p}(f)R\left(  \varphi_{\Psi},x,b\right)  ^{p}\right) \\
&  \leq\sum_{j=1}^{k}\sum_{(x,b)\in M_{j}}\left(  \left(  S(f,x,\lambda
_{j}^{\frac{1}{p}}b)\right)  ^{p}-\pi_{RS,p}(f)R\left(  \varphi_{\Psi
},x,\lambda_{j}^{\frac{1}{p}}b\right)  ^{p}\right) \\
&  \overset{(1)}{=}\sum_{(x,w)\in M_{\Psi}}\left(  S(f,x,w)^{p}-\pi
_{RS,p}(f)R\left(  \varphi_{\Psi},x,w\right)  ^{p}\right) \\
&  =\Psi_{M_{\Psi}}\left(  \varphi_{\Psi}\right)  .
\end{align*}
We have been considering finite sequences instead of finite sets precisely for equality (1) to hold true.
Using (\ref{cam-errado}) we obtain $\Psi_{M_{\Psi}}\left( \varphi_{\Psi}\right)  \leq0$ and therefore
\begin{equation}
\Psi\left(  \varphi_{\Psi}\right)  \leq0. \label{5}%
\end{equation}

Let%
\[
\mathcal{P}=\left\{  f\in C\left(  K\right)  ;f\left(  \varphi\right)
>0\text{ for all }\varphi\in K\right\}  .
\]
It is clear that $\mathcal{P}$ is non-void (every constant positive map
belongs to $\mathcal{P}$), open and convex. From the definition of
$\mathcal{P}$ and (\ref{5}) it follows that $\mathcal{P}\cap\mathcal{F}%
=\emptyset$. So, by Hahn-Banach Theorem there are $\mu_{1}\in C(K)^{\ast}$ and
$L>0$ such that
\begin{equation}
\mu_{1}\left(  g\right)  \leq L<\mu_{1}\left(  h\right)  \label{L}%
\end{equation}
for all $g\in\mathcal{F}$ and $h\in\mathcal{P}$.

If $x_{0}\in E$ is such that $R(\varphi,x_{0},b)=S(f,x_{0},b)=0$ for every
$\varphi\in K$ and $b\in G$, then
\[
0=S(f,x_{0},b)^{p}-\pi_{RS,p}(f)R\left(  \varphi,x_{0},b\right)  ^{p}%
=\Psi_{\left\{  (x_{0},b)\right\}  }\in\mathcal{F}\text{.}%
\]
Hence $0=\mu_{1}\left(  0\right)  \leq L$. As the constant functions $\frac {1}{k}$ belong to $\mathcal{P}$
for all $k\in\mathbb{N}$, it follows from (\ref{L}) that $0\leq L<\mu_{1}\left( \frac{1}{k}\right) $. Making
$k \longrightarrow + \infty$ we get $L=0$. Therefore
(\ref{L}) becomes%
\begin{equation}
\mu_{1}\left(  g\right)  \leq0<\mu_{1}\left(  h\right)  \label{0}%
\end{equation}
for all $g\in\mathcal{F}$ and all $h\in\mathcal{P}.$

Using the continuity of $\mu_{1},$ we conclude that $\mu_{1}\left(  h\right)
\geq0$ whenever $h\geq0$ and then we can consider $\mu_{1}$ a positive regular
Borel measure on $K$.

Defining
\[
\mu:=\frac{1}{\mu_{1}\left(  K\right)  }\mu_{1}%
\]
it is plain that $\mu$ is a regular probability measure on $K$, and from (\ref{0}) we get
\[
\int_{K}g\left(  \varphi\right)  \,d\mu\left(  \varphi\right)  \leq0
\]
for all $g\in\mathcal{F}$. In particular, for each $x,b$ we have
$\Psi_{\left\{  (x,b)\right\}  }\in\mathcal{F},$ and
\[
\int_{K}\Psi_{\left\{  (x,b)\right\}  }\left(  \varphi\right)  d\mu\left(
\varphi\right)  \leq0.
\]
So%
\[
\int_{K}\left(  S(f,x,b)^{p}-\pi_{RS,p}(f)R\left(  \varphi,x,b\right)
^{p}\right)  d\mu\left(  \varphi\right)  \leq0,
\]
and then%
\[
S(f,x,b)^{p}\leq\pi_{RS,p}(f)\int_{K}R\left(  \varphi,x,b\right)  ^{p}%
d\mu\left(  \varphi\right)  .
\]
Taking $p$-roots the result follows.
\end{proof}

\section{Recovering the known domination theorems}

In this section we show how Theorem \ref{Unified} can be easily invoked in order to obtain, as simple
corollaries, all known domination theorems (to the best of our knowledge) that have appeared in the several
different generalizations of the concept of absolutely $p$-summing linear operator. Given one of such classes
of {\it absolutely summing} mappings, it is easy to see that for convenient choices of $X$, $Y$, $E$, $G$,
$\mathcal{H}$, $K$, $R$ and $S$, for a mapping to belong to the class is equivalent to be $R$-$S$ abstract
summing mapping and that, in this case, the corresponding domination theorem that holds for this class is
nothing but Theorem \ref{Unified}.

\subsection{Pietsch's domination theorem for absolutely $p$-summing linear
operators}

\medskip Consider $X=E$ and $Y$ Banach spaces, $x_{0}=0$, $K=B_{E^{\prime}}$
and $G=\mathbb{K}$ (the scalar field). Take $\mathcal{H}=L(X;Y)$ the space of all linear operators from $X$
into $Y$ and define $R$ and $S$ by:
\[
R\colon B_{E^{\prime}}\times X\times\mathbb{K} \longrightarrow[0,\infty) ~,~
R(\varphi,x,\lambda)=\vert\lambda\vert\vert\varphi(x)\vert
\]
\[
S\colon L(X;Y)\times X\times\mathbb{K} \longrightarrow[0,\infty) ~,~
S(T,x,\lambda)=\vert\lambda\vert\Vert T(x)\Vert.
\]
With $R$ and $S$ so defined and any $0<p<\infty$, a linear operator $T\colon X\longrightarrow Y$ is $R$-$S$
abstract $p$-summing if and only if it is absolutely $p$-summing and Theorem \ref{Unified} becomes the
classical and well-known Pietsch domination theorem \cite[Theorem 2.12]{DJT}.

\subsection{The domination theorem for Lipschitz $p$-summing mappings}

Lipschitz $p$-summing mappings have been introduced by Farmer and Johnson \cite{FaJo} and are exactly the
$R$-$S$ abstract $p$-summing mappings for $X$ and $Y$ being metric spaces, $E=X\times X\times\mathbb{R}$,
$x_{0}=(x,x,0)$ (where $x$ is any point in $X$), $G=\mathbb{R}$, $K=B_{X}^{\#}$ the unit ball of the
Lipschitz dual $X^{\#}$ of $X$, $\mathcal{H}$ the set of all mappings from $X$ to $Y$ and $R$ and $S$ being
defined as follows:
\[
R\colon B_{X^{\#}} \times(X\times X\times\mathbb{R})\times\mathbb{R}
\longrightarrow[0,\infty) ~,~ R(f,(x,y,a),\lambda)=\vert a\vert^{1/p}%
\vert\lambda\vert\vert f(x)-f(y)\vert
\]
\[
S\colon{\mathcal{H}}\times(X\times X\times\mathbb{R})\times\mathbb{R}
\longrightarrow[0,\infty)~,~ S(T,(x,y,a),\lambda)=\vert a\vert^{1/p}%
\vert\lambda\vert\Vert T(x)-T(y)\Vert,
\]
where $\Vert T(x)-T(y)\Vert$ means the distance from $T(x)$ to $T(y)$ in $Y$.

In this context Theorem \ref{Unified} coincides with the domination theorem
for Lipschitz $p$-summing mappings \cite[Theorem 1(2)]{FaJo}.

\subsection{The domination theorem for dominated multilinear mappings and
polynomials}

Dominated multilinear mappings were introduced by Pietsch \cite{Pie} and dominated polynomials by Matos
\cite{anais}. Given $n\in \mathbb{N}$,
let $X_{1}, \ldots,X_{n}, Y$ be Banach spaces and consider $X=E=X_{1}%
\times\cdots\times X_{n}$, ${\mathcal{H}}=L(^{n}X_{1}, \ldots,X_{n};Y)$ the
space of all $n$-linear mappings from $X_{1}\times\cdots\times X_{n}$ to $Y$,
$x_{0}=(0,\ldots,0)$, $K=B_{(X_{1}\times\cdots\times X_{n})^{\prime}}$,
$G=\mathbb{K}$ and $R$ and $S$ defined by:

\noindent$R\colon B_{(X_{1}\times\cdots\times X_{n})^{\prime}}\times
(X_{1}\times\cdots\times X_{n})\times\mathbb{K} \longrightarrow[0,\infty),$
\[
\ \ \ \ R(\varphi,(x_{1},\ldots,x_{n}),\lambda)=\vert\lambda\vert\vert
\varphi(x_{1},\ldots,x_{n})\vert
\]

\noindent$S\colon L(^{n}X_{1},\ldots,X_{n};Y)\times(X_{1}\times\cdots\times
X_{n})\times\mathbb{K}\longrightarrow\lbrack0,\infty),$
\[
S(T,(x_{1},\ldots,x_{n}),\lambda)=|\lambda|\Vert T(x_{1},\ldots,x_{n}%
)\Vert^{1/n}.
\]
Then, given any $0<p<\infty$, an $n$-linear mapping  $T\colon X_{1}\times\cdots\times X_{n}\longrightarrow Y$
is $p$-dominated if and only if it is $R$-$S$ abstract $p$-summing. Moreover, in this setting Theorem
\ref{Unified} recovers the domination as it appears in \cite[Theorem 3.2(D)]{BPR}. In this same reference one
can learn how the latter domination theorem leads to the standard domination theorem for dominated
multilinear mappings that first appeared in \cite[Theorem 14]{Pie} (a direct
proof can be found in \cite[Satz 3.2.3]{Geiss}).\\
\indent For dominated polynomials, a repetition of the above application of Theorem \ref{Unified} to the
corresponding symmetric multilinear mappings yields the following characterization: an $n$-homogeneous
polynomial $P \colon X \longrightarrow Y$ is $p$-dominated if and only if there is a constant $C>0$ and a
regular Borel probability measure $\mu$ on $B_{(X^n)'}$ such that\begin{equation*} \|P(x)\|\leq C\left(
\int_{B_{(X^n)'}} |\varphi(x,\ldots,x)|^{p}d\mu\left( \varphi\right)  \right)  ^{\frac{n}{p}}
\label{2}\end{equation*} for every $x\in X$. As in \cite{BPR}, this characterization leads to the standard
domination theorem for dominated polynomials as it appears in \cite[Proposition 3.1]{anais}.

\subsection{The domination theorem for strongly $p$-summing
multilinear mappings and homogeneous polynomials}

Strongly $p$-summing mappings were introduced by Dimant \cite{Dimant}. For multilinear mappings, given
$n\in\mathbb{N}$, strongly $p$-summing $n$-linear mappings are particular cases of $R$-$S$ abstract
$p$-summing mappings considering $X_{1}, \ldots,X_{n}, Y$ Banach spaces, $X=E=X_{1}\times\cdots\times X_{n}$,
${\mathcal{H}}=L(^{n}X_{1}, \ldots,X_{n};Y)$ the space of all $n$-linear mappings from
$X_{1}\times\cdots\times X_{n}$ to $Y$, $x_{0}=(0,\ldots,0)$,
$K=B_{(X_{1}\hat\otimes_{\pi}\cdots\hat\otimes_{\pi}X_{n})^{\prime}}$, $G=\mathbb{K}$ and $R$ and $S$ defined
by:

\noindent$R\colon B_{(X_{1}\hat\otimes_{\pi}\cdots\hat\otimes_{\pi}%
X_{n})^{\prime}}\times(X_{1}\times\cdots\times X_{n})\times\mathbb{K}
\longrightarrow[0,\infty), $
\[
\ \ \ \ R(\varphi,(x_{1},\ldots,x_{n}),\lambda)=\vert\lambda\vert\vert
\varphi(x_{1}\otimes\cdots\otimes x_{n})\vert
\]

\noindent$S\colon L(^{n}X_{1}, \ldots,X_{n};Y)\times(X_{1}\times\cdots\times
X_{n})\times\mathbb{K} \longrightarrow[0,\infty), $
\[
S(T,(x_{1},\ldots,x_{n}),\lambda)=\vert\lambda\vert\Vert T(x_{1},\ldots
,x_{n})\Vert.
\]
In that case, Theorem \ref{Unified} recovers the corresponding domination
theorem \cite[Proposition 1.2(ii)]{Dimant}.

On the other hand, strongly $p$-summing $n$-homogeneous polynomials are obtained also as particular cases of
$R$-$S$ abstract
$p$-summing mappings. Just take $E=X, Y$ Banach spaces, ${\mathcal{H}%
}={P}(^{n}X;Y)$ the space of all $n$-homogeneous polynomials
from $X$ to $Y$, $x_{0}=0$, $K=B_{(\hat\otimes_{\pi}X)^{\prime}}%
=B_{{\mathcal{P}}(^{n}X)}$, $G=\mathbb{K}$ and $R$ and $S$ defined by:

\noindent$R\colon B_{{\mathcal{P}}(^{n}X)}\times X\times\mathbb{K}
\longrightarrow[0,\infty), $
\[
\ \ \ \ R(Q,x,\lambda)=\vert\lambda\vert\vert Q(x)\vert
\]

\noindent$S\colon{P}(^{n}X;Y)\times X\times\mathbb{K} \longrightarrow
[0,\infty), $
\[
S(P,x,\lambda)=\vert\lambda\vert\Vert P(x)\Vert.
\]
As expected, Theorem \ref{Unified} also recovers the corresponding domination theorem for strongly
$p$-summing $n$-homogeneous polynomials \cite[Theorem 3.2(ii)]{Dimant}.

\subsection{The domination theorem for $p$-semi-integral multilinear mappings}

The class of $p$-semi-integral multilinear mappings was introduced in \cite{Pell, CP} and are particular
cases of $R$-$S$ abstract $p$-summing mappings considering $n\in \mathbb{N}$, $X_{1},\ldots,X_{n},Y$ Banach
spaces, $E=X_{1}\times\cdots\times X_{n}$, ${\mathcal{H}}=L(^{n}X_{1},\ldots,X_{n};Y)$ the space of all
$n$-linear mappings from $X_{1}\times\cdots\times X_{n}$ to $Y$, $x_{0}=(0,\ldots,0)$,
$K=B_{X_{1}^{\prime}\times\cdots\times X_{n}^{\prime}}$, $G=\mathbb{K}$ and $R$ and $S$ defined by:

\noindent$R\colon B_{B_{X_{1}^{\prime}\times\cdots\times X_{n}^{\prime}}%
}\times(X_{1}\times\cdots\times X_{n})\times\mathbb{K}\longrightarrow
\lbrack0,\infty),$
\[
\ \ \ \ R(\varphi,(x_{1},\ldots,x_{n}),\lambda)=|\lambda||\varphi(x_{1}%
)\cdots\varphi(x_{n})|
\]

\noindent$S\colon L(^{n}X_{1},\ldots,X_{n};Y)\times(X_{1}\times\cdots\times
X_{n})\times\mathbb{K}\longrightarrow\lbrack0,\infty),$
\[
S(T,(x_{1},\ldots,x_{n}),\lambda)=|\lambda|\Vert T(x_{1},\ldots,x_{n})\Vert.
\]
In that case, Theorem \ref{Unified} recovers the corresponding domination theorem \cite[Theorem 1]{CP}.

\subsection{The domination theorem for $\tau(p)$-summing multilinear mappings}

The class of $\tau(p)$-summing multilinear mappings was introduced by X. Mujica in \cite{Xim} and is a
particular case of $R$-$S$ abstract $p$-summing mappings considering $n\in \mathbb{N}$, $X_{1},\ldots
,X_{n},Y$ Banach spaces, $E=X_{1}\times\cdots\times X_{n}$, ${\mathcal{H}%
}=L(^{n}X_{1},\ldots,X_{n};Y)$ the space of all $n$-linear mappings from
$X_{1}\times\cdots\times X_{n}$ to $Y$, $x_{0}=(0,\ldots,0)$, $K=B_{X_{1}%
^{\prime}}\times\cdots\times B_{X_{n}^{\prime}}\times B_{F^{\prime\prime}}$,
$G=F^{\prime}$ and $R$ and $S$ defined by:

\noindent$R\colon\left(  B_{X_{1}^{\prime}}\times\cdots\times B_{X_{n}%
^{\prime}}\right)  \times(X_{1}\times\cdots\times X_{n})\times F^{\prime
}\longrightarrow\lbrack0,\infty),$
\[
\ \ \ \ R((\varphi_{1},...,\varphi_{n},\varphi),(x_{1},\ldots,x_{n}%
),b)=|\varphi(x_{1})\cdots\varphi(x_{n})||\varphi(b)|
\]

\noindent$S\colon L(^{n}X_{1},\ldots,X_{n};Y)\times(X_{1}\times\cdots\times
X_{n})\times F^{\prime}\longrightarrow\lbrack0,\infty),$
\[
S(T,(x_{1},\ldots,x_{n}),\lambda)=\left\vert <b,T(x_{1},\ldots,x_{n}%
)>\right\vert .
\]
In that case, Theorem \ref{Unified} recovers the corresponding domination theorem \cite[Theorem 3.5]{Xim}.

\subsection{The domination theorem for subhomogeneous mappings}

The class of $\alpha$-subhomogeneous mappings was introduced in \cite{BPR} and
is a particular case of $R$-$S$ abstract $p$-summing mappings considering
$X,Y$ Banach spaces, $E=X$,
\[
{\mathcal{H}}=\{f:E\longrightarrow Y;\text{ }f\text{ is }\alpha
\text{-subhomogeneous and }f(0)=0\},
\]
$x_{0}=(0,\ldots,0)$, $K=B_{X^{\prime}}$, $G=\mathbb{K}$ and $R$ and $S$
defined by:

\noindent$R\colon B_{X^{\prime}}\times E\times\mathbb{K}\longrightarrow
\lbrack0,\infty),$
\[
\ \ \ \ R(\varphi,x,\eta)=\left\vert \eta\right\vert \left\vert \varphi
(x)\right\vert
\]

\noindent$S\colon{\mathcal{H}}\times E\times\mathbb{K}\longrightarrow
\lbrack0,\infty),$
\[
S(f,x,\eta)=\left\vert \eta\right\vert \left\Vert f(x)\right\Vert ^{1/\alpha}.
\]
In that case, Theorem \ref{Unified} recovers the corresponding domination theorem \cite[Theorem 2.4]{BPR}.

\subsection{The domination theorem for $(D,p)$-summing multilinear mappings}

The class of $(D,p)$-summing linear operators was introduced by F.
Mart\'{\i}nez-Gim\'{e}nez and E. A. S\'{a}nchez-P\'{e}rez in \cite{SP} and is
a particular case of $R$-$S$ abstract $p$-summing mappings considering $Y$ a
Banach space, $E$ a Banach function space compatible with a countably additive
vector measure $\lambda$ of the range dual pair $D=(\lambda,\lambda^{\prime}%
)$,
\[
{\mathcal{H}}=L(E;Y)
\]
the space of all linear mappings from $E$ to $Y$, $x_{0}=(0,\ldots,0)$,
$K=\overline{B}_{L_{1}(\lambda^{\prime})}$, $G=\mathbb{R}$ and $R$ and $S$
defined by:

\noindent$R\colon\overline{B}_{L_{1}(\lambda^{\prime})}\times E\times
\mathbb{R}\longrightarrow\lbrack0,\infty),$
\[
\ \ \ \ R(\phi,f,\eta)=\left\vert \eta\right\vert |(f,\phi)|
\]

\noindent$S\colon{\mathcal{H}}\times E\times\mathbb{R}\longrightarrow
\lbrack0,\infty),$
\[
S(T,f,\eta)=\left\vert \eta\right\vert \left\Vert T(f)\right\Vert .
\]
In that case, Theorem \ref{Unified} recovers the corresponding domination theorem \cite[Theorem 3.11]{SP}.

\section{Absolutely summing arbitrary mappings}

According to the usual definition of absolutely summing linear operators by means of inequalities, the
following definition is quite natural:

\begin{definition}\rm
\label{14nov08}Let $E$ and $F$ be Banach spaces. An arbitrary mapping $f\colon E\longrightarrow F$ is {\it
absolutely $p$-summing at $a\in E$} if there is a
$C\geq0$ so that%
\[%
{\displaystyle\sum\limits_{j=1}^{m}}
\left\Vert f(a+x_{j})-f(a)\right\Vert ^{p}\leq C\sup_{\varphi\in B_{E^{\prime
}}}%
{\displaystyle\sum\limits_{j=1}^{m}}
\left\vert \varphi(x_{j})\right\vert ^{p}%
\]
for every natural number $m$ and every $x_{1},\ldots,x_{m} \in E$. \bigskip
\end{definition}

As \cite[Theorem 3.5]{matos} makes clear, the above definition is actually an adaptation of
\cite[Definition 3.1]{matos}.\\
\indent We finish the paper applying Theorem \ref{Unified} once more to show that, even in the absence of
algebraic conditions, absolutely $p$-summing mappings are exactly those which enjoy a Pietsch-type
domination:
\begin{theorem}
Let $E$ and $F$ be Banach spaces. An arbitrary mapping $f\colon E\longrightarrow F$ is absolutely $p$-summing
at $a\in E$ if and only if there is a constant $C>0$
and a regular Borel probability measure $\mu$ on $B_{E^{\prime}}$ such that%
\begin{equation*}
\left\Vert f(a+x)-f(a)\right\Vert \leq C\left(  \int_{B_{E^{\prime}}%
}\left\vert \varphi(x)\right\vert ^{p}d\mu\left(  \varphi\right)  \right)
^{\frac{1}{p}}%
\end{equation*}
for all $x\in E.$
\end{theorem}

\begin{proof} Using a clever argument from \cite[p. 2]{FaJo} (also credited to M. Mendel and G.
Schechtman), applied by Farmer and Johnson in the context of Lipschitz summing mappings, one can see that $f$
is absolutely $p$-summing at $a$ if and only if there is a $C\geq0$
so that%
\begin{equation}%
{\displaystyle\sum\limits_{j=1}^{m}}
\left\vert b_{j}\right\vert \left\Vert f(a+x_{j})-f(a)\right\Vert ^{p}\leq
C\sup_{\varphi\in B_{E^{\prime}}}%
{\displaystyle\sum\limits_{j=1}^{m}}
\left\vert b_{j}\right\vert \left\vert \varphi(x_{j})\right\vert ^{p}
\label{chaNO}%
\end{equation}
for every $m\in\mathbb{N}$, $x_{1},\ldots,x_{m}\in E$ and scalars $b_{1},\ldots,b_{m}.$ The idea of the
argument is the following: by approximation, it is enough to deal with rationals $b_{i}$ and, by cleaning
denominators, we
can resume to integers $b_{i}.$ Then, for $b_{i}$ (integer) and $x_{1}%
,\ldots,x_{m}$, we consider the new collection of vectors in which each $x_{i}$
is repeated $b_{i}$ times.\\
\indent Putting $X=E,Y$ $=F$,
\[
{\mathcal{H}}=\left\{  f \colon E\longrightarrow F\right\}  ,
\]
$x_{0}=0$, $K=B_{E^{\prime}}$, $G=\mathbb{K}$ and $R$ and $S$ defined by:

\noindent$R\colon B_{E^{\prime}}\times E\times\mathbb{K}\longrightarrow
\lbrack0,\infty),$
\[
\ \ \ \ R(\varphi,x,\lambda)=|\lambda||\varphi(x)|
\]

\noindent$S\colon{\mathcal{H}}\times E\times\mathbb{K}\longrightarrow
\lbrack0,\infty),$
\[
S(f,x,\lambda)=|\lambda|\Vert f(a+x)-f(a)\Vert,
\]
in view of characterization (\ref{chaNO}) it follows that $f$ is absolutely $p$-summing if and only if $f$ is
$R$-$S$ abstract $p$-summing. So, Theorem \ref{Unified} completes the proof. \end{proof}

\bigskip\noindent[Geraldo Botelho] Faculdade de Matem\'atica, Universidade
Federal de Uberl\^andia, 38.400-902 - Uberl\^andia, Brazil, e-mail: botelho@ufu.br.

\medskip

\noindent[Daniel Pellegrino] Departamento de Matem\'atica, Universidade
Federal da Para\'iba, 58.051-900 - Jo\~ao Pessoa, Brazil, e-mail: dmpellegrino@gmail.com.

\medskip

\noindent[Pilar Rueda] Departamento de An\'alisis Matem\'atico, Universidad de
Valencia, 46.100 Burjasot - Valencia, Spain, e-mail: pilar.rueda@uv.es.

\end{document}